\documentclass[12pt]{amsart}
\usepackage{amssymb,latexsym}
\usepackage[english]{babel}
\usepackage[pdftex,colorlinks=true]{hyperref}
\usepackage[T1]{fontenc}
\usepackage{XCharter}\usepackage{eulervm}

\theoremstyle{plain}
\newtheorem{thm}{Theorem}[section]
\newtheorem{lem}[thm]{Lemma}
\newtheorem{cor}[thm]{Corollary}
\newtheorem*{add}{Additivity property}
\newtheorem*{otopy_in}{Otopy invariance property}
\newtheorem*{existence}{Existence property}
\newtheorem*{normal}{Normalization property}
\newtheorem*{product}{Product property} 

\theoremstyle{definition}
\newtheorem{defn}[thm]{Definition}

\theoremstyle{remark}
\newtheorem{rem}[thm]{Remark}

\numberwithin{equation}{section}

\newcommand{\wt}{\widetilde}
\newcommand{\R}{\mathbb{R}}
\newcommand{\N}{\mathbb{N}}

\newcommand{\G}{\mathcal{G}}
\newcommand{\fU}{\mathfrak{U}}
\newcommand{\restrictionmap}[2]{{#1}\mathpunct\restriction\hbox{}_{#2}}
\providecommand{\abs}[1]{\left\lvert#1\right\rvert}

\DeclareMathOperator{\id}{Id}
\DeclareMathOperator{\Deg}{Deg}
\DeclareMathOperator{\cl}{cl}

\title[Topological degree]
{Topological degree for equivariant gradient perturbations
of an unbounded self-adjoint operator in Hilbert space} 

\author[P. Bart{\l}omiejczyk]{Piotr Bart{\l}omiejczyk}
\address{Faculty of Applied Physics and Mathematics,
Gda{\'n}sk University of Technology,
Gabriela Narutowicza 11/12,
80-233 Gda{\'{n}}sk, Poland}
\email{piobartl@pg.edu.pl}
\author[B. Kamedulski]{Bartosz Kamedulski}
\address{Faculty of Navigation, Gdynia Maritime University, 
Jana Paw{\l}a II 3, 81-345 Gdynia, Poland}
\email{b.kamedulski@wn.umg.edu.pl}
\author[P. Nowak-Przygodzki]{Piotr Nowak-Przygodzki}
\address{Faculty of Applied Physics and Mathematics,
Gda{\'n}sk University of Technology,
Gabriela Narutowicza 11/12,
80-233 Gda{\'{n}}sk, Poland}
\email{piotrnp@wp.pl}

\date{\today}

\begin{document}

\begin{abstract}
We present a version of 
the equivariant gradient degree defined
for equivariant gradient perturbations
of an equivariant unbounded self-adjoint operator 
with purely discrete spectrum in Hilbert space.
Two possible applications are discussed.
\end{abstract}

\maketitle


\section*{Introduction}\label{sec:intro}

To obtain new bifurcation results, 
N. Dancer \cite{D} introduced in 1985 a new topological invariant
for $S^1$-equivariant gradient maps, 
which provides more information than the usual equivariant one. 
In 1994 S.~Rybicki \cite{Ry1,Ry3} developed the complete
degree theory for $S^1$-equi\-variant gradient 
maps and 3 years later K. G\k{e}ba extended 
this theory to an arbitrary compact Lie group. 
In 2001 S.~Rybicki \cite{Ry2} defined the degree for 
$S^1$-equivariant strongly indefinite functionals in Hilbert space. 
10 years later A.~Go{\l}\k{e}biewska and S.~Rybicki \cite{GR1}
generalized this degree to compact Lie groups. 
The relation between equivariant and equivariant 
gradient degree theories were studied in~\cite{BGI,BP,GI}.

The main goal of this paper is 
to present a construction and properties of
a new degree-type topological invariant $\Deg_G^\nabla$,
which is defined for equivariant gradient perturbations of 
a equivariant unbounded self-adjoint 
Hilbert operator with a purely discrete spectrum 
(in the general case a compact Lie group).
As far as we know, 
the idea of the construction of such an invariant 
should be attributed to K. G\k{e}ba.

It is worth pointing out that
equivariant gradient perturbations of an
equivariant unbounded self-adjoint operator with a purely
discrete spectrum appear naturally in a variety 
of problems in nonlinear analysis, such as the search 
for periodic solutions of Hamiltonian systems or
the study of Seiberg-Witten equations for three
dimensional manifolds. The purpose of our work
is to provide a topological tool 
that allows us to solve problems
similar to the above mentioned ones. 

The paper is organized as follows.
Section~\ref{sec:prel} contains some preliminaries.
In Section~\ref{sec:degree} we present 
the construction that leads to 
the definition of the degree $\Deg_G^\nabla$.
The correctness of this definition is proved
in Section~\ref{sec:correct}.
The properties of the degree $\Deg_G^\nabla$
are examined in Section~\ref{sec:properties}.
Finally, in Section~\ref{sec:applic} we discuss
two examples of possible applications.


\section{Preliminaries}\label{sec:prel}
The preliminaries are divided into five brief subsections.

\subsection{Unbounded self-adjoint operators in Hilbert space}
This subsection is based on \cite{Sch}.
Let $E$ be a real separable Hilbert space with
inner product $\langle\cdot\mid\cdot\rangle$
and $A\colon D(A) \subset E \to E$ be a linear operator
(not necessarily bounded)
such that its domain $D(A)$ is dense in $E$. Set 
\[
D(A^*)=\{y \in E\mid \exists u \in E\,
\forall x\in D(A)\; 
\langle Ax\mid y\rangle=\langle x\mid u \rangle\}.
\]
Since $D(A)$ is dense in $E$, 
the vector $u \in E$ is uniquely determined by~$y$. 
Therefore by setting $A^*y=u$ we obtain 
a well-defined linear operator from $D(A^*)$ to $E$. 
The operator $A^*$ is called 
the \emph{adjoint} operator of~$A$.
We say that $A$ is \emph{self-adjoint} if $A=A^*$.
By the Hellinger-Toeplitz theorem, if $A$ is self-adjoint
and $D(A)=E$ then $A$ is bounded.

It is easy to see that
\[
\langle x\mid y\rangle_1=
\langle x\mid y\rangle+
\langle Ax\mid Ay\rangle
\]
defines an inner product on the domain $D(A)$.
Under this product $D(A)$ becomes a Hilbert space, 
which will be denoted by $E_1$.
Thus $D(A)$ and $E_1$ are equal as sets 
but equipped with different inner products.
Note that $A$ treated as an operator from 
$E_1$ to $E$ is bounded.
 
We say that a self-adjoint operator $A$ 
has a \emph{purely discrete spectrum}
if its spectrum consists only of 
isolated eigenvalues of finite multiplicity.
If $E$ is an infinite dimensional Hilbert space
then following conditions are equivalent:
\begin{enumerate}
	\item $A$ has a purely discrete spectrum.
	\item There is a real sequence $\{\lambda_n\}$ and an
	orthonormal basis $\{e_n\}$ such that
	$\lim\abs{\lambda_n}=\infty$ and $Ae_n=\lambda_n$
	for $n\in\N$.
	\item The embedding $\imath\colon E_1\to E$ is compact.
\end{enumerate}

\subsection{Local maps in Hilbert space}

Let
\begin{itemize}
	\item $E$ be a real Hilbert orthogonal representation
   of a compact Lie group $G$,
	\item $A\colon D(A)\subset E\to E$ 
  be an unbounded self-adjoint operator 
	with a purely discrete spectrum,
	\item $D(A)$ be invariant and $A$ equivariant.
\end{itemize}
 
\begin{defn}
We write $f\in\mathcal{G}_{G}(E)$ if
\begin{itemize}
	\item $f\colon D_f\subset E_1\to E$,
	where $D_f$ is an open invariant subset of $E_1$,
	\item $f(x)=Ax-\nabla\varphi(x)$, 
	where $\varphi\colon E\to\mathbb{R}$
	is $C^1$ and invariant,
	\item $f^{-1}(0)$ is compact.
\end{itemize}
\end{defn}
Elements of $\mathcal{G}_{G}(E)$ will be called \emph{local maps}.	

\subsection{Otopies in Hilbert space}
Let $I=[0, 1]$. Assume that $G$ acts trivially on $I$.
A map $h\colon\Lambda\subset I\times E_1\to E$ 
is called an \emph{otopy} if
\begin{itemize}
	\item $\Lambda$ is an open invariant subset of $I\times E_1$,
	\item $h(t,\cdot)\in \mathcal{G}_{G}(E)$ for each $t\in I$,
	\item $h^{-1}(0)$ is compact.
\end{itemize}

Given an otopy
$h\colon\Lambda\subset I\times E_1\to E$ 
we can define for each  $t\in I$:
\begin{itemize}
	\item sets $\Lambda_t=\{x\in E_1\mid(t,x)\in\Lambda\}$,
	\item maps $h_t\colon\Lambda_t\to E$ with $h_t(x)=h(t,x)$.
\end{itemize}
If $h$ is an otopy, we say that  
$h_0$ and $h_1$ are \emph{otopic}.
The relation of being otopic is
an equivalence relation in $\G_G(E)$.

Observe that if $f$ is a local map and $U$ is 
an open subset of $D_f$ such that $f^{-1}(0)\subset U$, 
then $f$ and $\restrictionmap{f}{U}$ are otopic.
This property of local maps
is called the \emph{restriction property}.
In particular, if $f^{-1}(0)=\emptyset$ then $f$ 
is otopic to the empty map.

\subsection{Euler-tom Dieck ring}
Recall the notion of the \emph{Euler-tom Dieck ring}
following \cite{T}.
For a compact Lie group $G$ 
let $\mathfrak{U}(G)$ denote the set 
of equivalence classes of finite $G$-CW-complexes.
Two complexes $X$ and $Y$ are identified
if the quotients $X^H/WH$ and  $Y^H/WH$
have the same Euler characteristic for
all closed subgroups $H$ of $G$.
Recall that $X^H$ stands here for
the $H$-fixed point set of $X$, i.e.
$X^H:=\{x\in X\mid hx=x\text{ for all $h\in H$}\}$
and $WH$ for the Weyl group of $H$, i.e. $WH=NH/H$.
Addition and multiplication in $\mathfrak{U}(G)$
are induced by disjoint union and cartesian product 
with diagonal $G$-action, i.e.\
\[
[X]+[Y]=[X\sqcup Y],\quad[X]\cdot[Y]=[X\times Y],
\] 
where the square brackets stand for an equivalence class of 
finite $G$-CW-complexes.
In this way $\mathfrak{U}(G)$ becomes a commutative ring 
with unit and is called the \emph{Euler-tom Dieck ring} of $G$.

Additively, $\mathfrak{U}(G)$ is a free abelian group 
with basis elements $[G/H]$,
where $H$ is a closed subgroup of $G$.
In consequence, each element of $\mathfrak{U}(G)$ 
can be uniquely written
as a finite sum $\sum d_{(H)}[G/H]$, 
where $d_{(H)}$ is an integer, 
which depends only on the conjugacy class of $H$. 
The ring unit is $[G/G]$.

\subsection{Finite dimensional equivariant gradient degree 
\texorpdfstring{$\deg^\nabla_G$}{}}
Assume that $V$ is a real finite dimensional orthogonal
representation of a compact Lie group $G$.
We write $f\in\G_G(V)$ 
if $f$ is an equivariant gradient map 
from an open invariant subset of $V$ to $V$ 
and $f^{-1}(0)$ is compact.
In the papers \cite{BGI,BP,G,Ry3} the authors defined
the equivariant gradient degree 
\[
\deg^\nabla_G\colon\G_G(V)\to\fU(G)
\]
and proved that the degree has the following properties:
additivity, otopy invariance, existence and normalization.
The product property formulated below was proved in 
\cite{G} and \cite{GR2}.
\begin{thm}[Product property]\label{thm:prod}
Let $V$ and $W$ be real finite dimensional orthogonal
representations of  a compact Lie group $G$.
If $f\in\G_G(V)$ and $f'\in\G_G(W)$, 
then $f\times f'\in\G_G(V\oplus W)$ and
\[
\deg^\nabla_G(f\times f')=
\deg^\nabla_G(f)\cdot\deg^\nabla_G(f')
\text{ in $\fU(G)$.}
\]
\end{thm}
In the next section we will make use of the following result, 
which can be found in \cite[Cor. 2.1]{GR1}.
\begin{thm}\label{thm:goryb}
Let $V$ be a real finite dimensional orthogonal representation
of a compact Lie group $G$. If $B$ is an equivariant 
self-adjoint isomorphism of $V$ then $\deg^\nabla_{G}(B)$
is invertible in $\fU(G)$.
\end{thm}

\begin{rem}
Note that Theorem \ref{thm:goryb} 
holds even if $V$ is trivial.
In this case $\deg^\nabla_{G}(B)$ 
is equal to the unit of $\fU(G)$.
\end{rem}


\section{Definition of degree}\label{sec:degree}
In this section we present the construction of the degree 
$\Deg^\nabla_G$ using finite dimensional approximations.

\subsection{Finite dimensional approximations}
Let us start with some notations:
\begin{itemize}
	\item for $\lambda\in\sigma(A)$ denote by $V(\lambda)$ 
	the corresponding eigenspace;
	\item for $n\in\N$ write 
	$\displaystyle V_n=\oplus_{\abs{\lambda}\le n}V(\lambda)$,
	$\displaystyle V^n=\oplus_{n-1<\abs{\lambda}\le n}V(\lambda)$ 
	and $A_n=\restrictionmap{A}{V^n}$;
	hence $V_n=V_{n-1}\oplus V^n$;
	\item let $P_n\colon E\to V_n$ denote the orthogonal projection.
\end{itemize}
Assume that $U$ is an open bounded invariant subset 
of $D_f$ such that
  \[f^{-1}(0)\subset U\subset\cl U\subset D_f.\]
Set $U_n=U\cap V_n$. 
Finally, let $f_n\colon U_n\to V_n$ be given by
\[f_n(x)=Ax-P_nF(x),\]
where $F(x)=\nabla\varphi(x)$.

The following two lemmas are needed 
to prove Lemma~\ref{lem:defn},
which is crucial for the definition of $\Deg^\nabla_G$.

\begin{lem}\label{lem:epsilon}
There is $\epsilon>0$ such that $\abs{f(x)}\ge2\epsilon$
for all $x\in\partial U$.
\end{lem}

\begin{proof}
The fact $F$ is compact and $\partial U$ is 
closed and bounded implies our claim.
\end{proof}

Let us introduce an auxiliary map
$\wt{f}_n\colon D_f\to E$ 
given by $\wt{f}_n(x)=Ax-P_nF(x)$.
By definition, $\restrictionmap{\wt{f}_n}{U_n}=f_n$.

\begin{lem}\label{lem:close}
There is $N$ such that for $n\ge N$ we have
\begin{enumerate}
	\item $\lvert f(x)-\wt{f}_n(x)\rvert<\epsilon$ 
	for $x\in\cl U$,
	\item $\lvert\wt{f}_n(x)\rvert>\epsilon$ 
	for $x\in\partial U$.
\end{enumerate}
\end{lem}

\begin{proof}
Since $F$ is compact, $F$ is close to $P_nF$, which gives (1).
In turn (2) follows from (1) and Lemma \ref{lem:epsilon}.
\end{proof}

\begin{lem}\label{lem:defn}
For $n\ge N$ we have $f_n\in\G_G(V_n)$ and, in consequence,
$\deg^\nabla_G(f_n)\in \fU(G)$ is well-defined.
\end{lem}

\begin{proof}
Since $f_n$ is obviously gradient,
it is enough to check that $f_n^{-1}(0)$ is compact.
Note that $\wt{f}_n$ can be considered as an extension 
of $f_n$ on $\cl U_n$. 
By~(2) from Lemma~\ref{lem:close},
$\wt{f}_n$ does not have zeroes in 
$\partial U_n\subset\partial U$,
which implies that $f_n^{-1}(0)=\wt{f}_n^{-1}(0)\cap U_n$
is compact.
\end{proof}

\subsection{Degree definition}
Observe that $A_n$ is an equivariant self-adjoint 
isomorphism for $n\ge 1$. By Theorem \ref{thm:goryb},
elements $a_n:=\deg^\nabla_G(A_n)$
are invertible in $\fU(G)$.
Set $m_n:=a_1^{-1}\cdot a_2^{-1}\cdot\dotsb\cdot a_n^{-1}$.

\begin{defn}\label{defn:deg}	
Let $\Deg^\nabla_{G}\colon\mathcal{G}_{G}(E)\to\fU(G)$
be defined by
\[
\Deg^\nabla_{G}(f):=
m_n\cdot\deg^\nabla_G(f_n)
\]
for $n\ge N$. 
\end{defn}

An alternative definition of $\Deg^\nabla_{G}$ in terms 
of the direct limit is given in Appendix \ref{sec:dodatekA}.


\section{Correctness of the definition}\label{sec:correct}

We have to prove that our definition does not depend on 
the choice of $n$ and the neighbourhood $U$. 

\subsection{Independence from the choice 
of \texorpdfstring{$n$}{n}}
To show this we will need
the following lemma.

\begin{lem}\label{lem:susp}
For $n$ large enough
$f_{n+1}$ is otopic to 
$f_n\times A_{n+1}$ 
in $\mathcal{G}_G(V_{n+1})$ and hence
\[
\deg^\nabla_{G}(f_{n+1})=
\deg^\nabla_{G}(f_{n}\times A_{n+1}).
\]
\end{lem}

\begin{proof}
First observe there is an open  
$W \subset U$ and natural number $N$ such that
\begin{itemize}
    \item $f^{-1}(0) \subset W \subset U$,
    \item $P_n(\cl W) \subset U_n$ for all $n \geq N.$
\end{itemize}
Define $h_{n+1}\colon I \times \cl W_{n+1} \to V_{n+1}$ by
\[
h_{n+1}(t,x)=(1-t)f_{n+1}(x)+t(f_n \times A_{n+1})(x).
\]
We set $n$ sufficiently large.
One can show that $h_{n+1}(t,x)\neq 0$ 
for $t \in I$ and $x \in \partial W_{n+1}$. In consequence,  
$\restrictionmap{h_{n+1}}{I \times W_{n+1}}$ 
is a finite dimensional equivariant gradient otopy 
between $\restrictionmap{f_{n+1}}{W_{n+1}}$ and 
$\restrictionmap{f_n \times A_{n+1}}{W_{n+1}}$ 
(otherwise there would be a point $x_0 \in \partial W$ 
such that $f(x_0)=0$, a contradiction).
On the other hand, by the restriction property,
$f_{n+1}$ and $f_n \times A_{n+1}$ 
are otopic to their restrictions to $W_{n+1}$, 
which completes the proof.
\end{proof}

From Lemma \ref{lem:susp} and Theorem \ref{thm:prod}
we can easily conclude that
\begin{multline*}
\deg^\nabla_{G}(f_{n+1})
\stackrel{\ref{lem:susp}}{=}
\deg^\nabla_{G}(f_{n}\times A_{n+1})
\stackrel{\ref{thm:prod}}{=}\\
\deg^\nabla_{G}(f_{n})\cdot
\deg^\nabla_{G}(A_{n+1})=
a_{n+1}\cdot\deg^\nabla_{G}(f_{n}).
\end{multline*}
This gives
\[
m_{n+1}\cdot\deg^\nabla_G(f_{n+1})=
m_{n+1}\cdot a_{n+1}\cdot\deg^\nabla_G(f_{n})=
m_{n}\cdot\deg^\nabla_G(f_{n}),
\]
which shows that $\Deg^\nabla_{G}(f)$ does not 
depend on the choice of $n$ large enough.

\subsection{Independence from the choice 
of \texorpdfstring{$U$}{U}}
According to our definition 
$\Deg^\nabla_{G}(f)=\Deg^\nabla_G(\restrictionmap{f}{U})$.
Now we will prove that in fact $\Deg^\nabla_{G}(f)$
is independent from the choice of the neighbourhood $U$.

\begin{lem}
Let $W$ and $U$ be open bounded sets such that 
\[
f^{-1}(0) \subset W \subset U \subset \cl U \subset D_f.
\] 
Then $\Deg^\nabla_G(\restrictionmap{f}{W})
=\Deg^\nabla_G(\restrictionmap{f}{U})$.
\end{lem}

\begin{proof}
By the analogue of Lemma \ref{lem:epsilon}
(with $\partial U$ replaced by $\cl U \setminus W$),
$\abs{f(x)}\ge2\epsilon$ for $x\in\cl U \setminus W$ and
by Lemma \ref{lem:close}, 
$\lvert f(x)-\wt{f}_n(x)\rvert< \epsilon$ 
for $x \in \cl U$. Hence $\wt{f}_n(x) \neq 0$ 
for $x \in \cl U \setminus W$. 
In consequence, $f_n(x) \neq 0$ 
for $x \in \cl U_n \setminus W_n$. 
Therefore
\[
\Deg^\nabla_G(\restrictionmap{f}{U})=
m_n\cdot\deg^\nabla_G(\restrictionmap{f_n}{U_n})=
m_n\cdot\deg^\nabla_G(\restrictionmap{f_n}{W_n})=
\Deg^\nabla_G(\restrictionmap{f}{W}).
\qedhere
\]
\end{proof}

\begin{cor}
Let $U$ and $U'$ be open bounded subsets of $D_f$ such that
\[
f^{-1}(0) \subset U \cap U' \subset \cl(U \cup U') \subset D_f.
\]
Then $\Deg^\nabla_G(\restrictionmap{f}{U})=
\Deg^\nabla_G(\restrictionmap{f}{U \cap U'})=
\Deg^\nabla_G(\restrictionmap{f}{U'})$.  
\end{cor}

In this way we have proved that $\Deg^\nabla_G(f)$ 
does not depend on the choice of admissible $U$.


\section{Degree properties}\label{sec:properties}

In this section we prove that our degree 
$\Deg^\nabla_G\colon\G_G(E)\to \mathfrak{U}(G)$
has all properties analogous to the well-known
properties of the finite dimensional
equivariant gradient degree $\deg^\nabla_G$.

\begin{add}
If $f,f'\in\G_G(E)$ and $D_f\cap D_{f'}=\emptyset$ then
\[
\Deg^\nabla_G(f\sqcup f')=\Deg^\nabla_G(f)+\Deg^\nabla_G(f').
\]
\end{add}

\begin{otopy_in}
Let $f,f'\in\G_G(E)$. If $f$ is otopic to $f'$ then 
\[
\Deg^\nabla_G(f) = \Deg^\nabla_G(f').
\]
\end{otopy_in}

\begin{existence}
If $\Deg^\nabla_G(f) \neq 0$ then $f(x)=0$ for some $x \in D_f$.
\end{existence}

\begin{normal}
\[
\Deg^\nabla_G(A+P_0)=[G/G]=1_{\fU(G)},
\] 
where $P_0: E_1 \to V_0=\ker A$ is the orthogonal projection.
\end{normal}

\begin{product}
Let $E$ and $E'$ be real Hilbert orthogonal 
representations of a compact Lie group $G$.
If $f\in\G_G(E)$ and $f'\in\G_G(E')$, 
then $f\times f'\in\G_G(E\oplus E')$ and
\[
\Deg^\nabla_G(f\times f')=\Deg^\nabla_G(f)\cdot\Deg^\nabla_G(f'),
\]
where the dot here denotes the multiplication in $\fU(G)$.
\end{product}

\noindent\emph{Proof.}
\subsection*{Additivity}
Immediately from the additivity 
of $\deg^\nabla_G$ we obtain
\begin{multline*}
\Deg^\nabla_G(f\sqcup f')=
m_n\cdot\deg^\nabla_G(f_n\sqcup f'_n)=\\
m_n\cdot(\deg^\nabla_G(f_n)+\deg^\nabla_G(f'_n))=
\Deg^\nabla_G(f)+\Deg^\nabla_G(f').
\end{multline*}

\subsection*{Otopy invariance}
Let the map $h\colon\Lambda\subset I\times E_1\to E$
given by $h(t,x)=Ax-F(t,x)$ be an otopy.
We introduce the following notation:
\begin{align*}
\Lambda^t=&\{x\in E_1\mid(t,x)\in\Lambda\},
&h^t&\colon\Lambda^t\to E,
&h^t(x)&=h(t,x),\\
\Lambda_n=&\Lambda\cap(I\times V_n),
&h_n&\colon\Lambda_n\to V_n,
&h_n(t,x)&=Ax-P_nF(t,x),\\
\Lambda_n^t=&\Lambda^t\cap V_n,
&h_n^t&\colon\Lambda_n^t\to V_n,
&h_n^t(x)&=h_n(t,x).
\end{align*}
Note that for the needs of this subsection
the time parameter $t$ of the otopy is
a superscript, not a subscript.
According to the above notation
we have to show that $\Deg^\nabla_G(h^0)=\Deg^\nabla_G(h^1)$.
Since $h^{-1}(0)$ is compact, there is an open bounded set 
$W \subset I \times E_1$ such that 
\[
    h^{-1}(0) \subset W \subset \cl W \subset \Lambda.
\]
Hence for $i=0, 1$ we have
\[
(h^i)^{-1}(0) \subset W^i 
\subset \cl W^i \subset \Lambda^i,
\]
where $W^i=\{x\in E_1 \mid (i,x) \in W\}$.
Similarly as in Lemma \ref{lem:epsilon}, 
there is $\epsilon>0$ such that $\abs{h(z)}\ge 2 \epsilon$ 
for $z \in \partial W$. On the other hand, 
similarly as in Lemma \ref{lem:close}, 
there is $N$ such that 
$\big\lvert h(z)-\wt{h}_n(z)\big\rvert<\epsilon$ for 
$z \in \cl W$ and $n \ge N,$ where $\wt{h}_n: \Lambda \to E$ 
is given by $\wt{h}_n(t,x)=Ax-P_nF(t,x)$.
Therefore $\abs{h_n(z)}\ge\epsilon$ 
for $z\in\partial W_n\subset\partial W$.
From the above:
\begin{itemize}
    \item $\restrictionmap{h_n}{W_n}$ is 
		a finite dimensional equivariant gradient otopy,
    \item $\Deg^\nabla_G(h^i)=
		m_n\cdot\deg^\nabla_G(\restrictionmap{h_n^i}{W_n^i})$,
\end{itemize}
which, by the otopy invariance of $\deg^\nabla_G$, gives
\[
\Deg^\nabla_G(h^0)=
m_n\cdot\deg^\nabla_G(\restrictionmap{h_n^0}{W_n^0})=
m_n\cdot\deg^\nabla_G(\restrictionmap{h_n^1}{W_n^1})=
\Deg^\nabla_G(h^1).
\]

\subsection*{Existence}
If $f^{-1}(0)=\emptyset$ then $f$ is 
otopic with the empty map. Hence 
\[
\Deg^\nabla_G(f)=\Deg^\nabla_G(\emptyset)=0.
\]

\subsection*{Normalization} 
Observe that $A+P_0$ is an injection and
\[
\deg^\nabla_G((A+P_0)_n)=
\deg^\nabla_G(\restrictionmap{\id}{V_0})\cdot 
\deg^\nabla_G(A_1)\cdot
\ldots\cdot\deg^\nabla_G(A_n)=m_n^{-1}
\]
for any $n\ge1$. Hence
\[
\Deg^\nabla_G(A+P_0)=
m_n\cdot\deg^\nabla_G((A+P_0)_n)=[G/G].
\]

\subsection*{Product formula}
Let $f(x)=Ax-F(x)$ and $f'(x)=A'x-F'(x)$.
Observe that, by Theorem~\ref{thm:prod}, 
if $f_n\in\G_G(V_n)$ and $f_n'\in\G_G(V_n')$
then $f_n\times f_n'\in\G_G(V_n\oplus V_n')$ and
\[
\deg^\nabla_G(f_n\times f_n')=
\deg^\nabla_G(f_n)\cdot
\deg^\nabla_G(f_n').
\]
Moreover, for $n$ large enough
\begin{align*}
\Deg^\nabla_G(f)&=m_n\cdot\deg^\nabla_G(f_n),\\
\Deg^\nabla_G(f')&=m_n'\cdot\deg^\nabla_G(f_n').
\end{align*}
Since for any $i\ge1$
\[
\deg^\nabla_G((A \times A')_i)=
\deg^\nabla_G(A_i\times A_i')=
\deg^\nabla_G(A_i)\cdot\deg^\nabla_G(A_i'),
\]
we have
\begin{multline*}
\Deg^\nabla_G(f\times f')=
m_n\cdot m_n'\cdot\deg^\nabla_G(f_n \times f_n')=\\
m_n\cdot m_n'\cdot\deg^\nabla_G(f_n)
\cdot\deg^\nabla_G(f_n')=
\Deg^\nabla_G(f)\cdot \Deg^\nabla_G(f').
\end{multline*}

\qed

\begin{rem}
The normalization property 
can be formulated more generally,
but the proof of 
this fact will appear elsewhere.
Namely, let $x_0\in V_n$ and, 
in consequence, $G{x_0}\subset V_n$. 
Define
\begin{multline*}
U=\{x+y+z\mid x\in G{x_0},\; 
y\in\big(T_{x_0}(G{x_0})\big)^{\bot}\subset V_n,\;\\
\abs{y}<\epsilon,\;
z\in\big(V_n\big)^{\bot}\subset E_1\}
\end{multline*}
and $f\colon U\to E$ by
\[
f(x+y+z)=(A+P_0)(y+z).
\]
Then $\Deg^\nabla_G(f)=[G/G_{x_0}]$.
\end{rem}


\section{Possible applications}\label{sec:applic}
We should emphasize that this section contains 
not real applications of the theory but only two 
exemplary situations illustrating potential applications.

\subsection{Applications to Hamiltonian systems}
The search for periodic solutions in Hamiltonian systems
is one of the fundamental problems in nonlinear analysis
(see for instance \cite{BS,Ra1,Ra2,W}).
Consider the Hamiltonian
system of	ODE
\[
\frac{dp}{dt}=-H_q,\qquad
\frac{dq}{dt}=H_p,
\]
where $H\in C^1(\R^{2n},\R)$ 
and $p,q\in\R^n$ or equivalently
\[
\frac{dz}{dt}=\mathcal{J}H_z,
\]
where $z=(p,q)$ and
\[
\mathcal{J}=
\begin{pmatrix}
0& -I\\
I& \;\;\;0
\end{pmatrix}.
\]
The function $H$ is called the hamiltonian or energy.

Rewrite the Hamiltonian system as
\[\tag{$*$}
\dot{z}=\mathcal{J}\nabla H(z), \quad z\in\mathbb{R}^{2n}
\]
or equivalently
$-\mathcal{J}\dot{z}-\nabla H(z)=0$.

We are searching for solutions $z\in H_T^1$ 
of the equation ($*$),
where $H_T^1$ ($T>0$) denotes 
the completion of the set of smooth $T$-periodic
functions from $\R$ to $\R^{2n}$ 
in the norm associated to the inner product
$(u\,\vert\, v)_{H_T^1}=
\int_0^T uv\,dt+\int_0^T \dot{u}\dot{v}\,dt$.
For this purpose we apply the method 
of the topological degree $\Deg^\nabla_{S^1}$.
Namely, let $E=L^2(S^1,\mathbb{R}^{2n})$ 
and $E_1=H^1(S^1,\mathbb{R}^{2n})$.
Moreover, denote by $D$ the set $E_1$ 
equipped with the inner product from $E$.

Observe that
\begin{itemize}
	\item $E$ and $E_1$ are Hilbert spaces and orthogonal
	representations of the group $SO(2)=S^1$
	with the $S^1$-action given by the shift in time,
	\item $A\colon D\to E$ given by $Az=-\mathcal{J}\dot{z}$ 
  is an equivariant unbounded self-adjoint operator 
	with a purely discrete spectrum, 
	\item $\nabla H(z)$ is a gradient of the invariant functional
  $\varphi\colon E\to\R$ defined by 
	$\varphi(z)=\int_0^{2\pi}H(z(t))\,dt$,
  \item $\nabla H\circ\imath\colon E_1\to E$ is a compact map
	by the compactness of the~inclusion $\imath\colon E_1\to E$.	
\end{itemize}

We can now formulate the main result of this subsection.

\begin{thm}
Assume that $\lambda>0$ and the set of zeros of the map 
$f_\lambda(z)=-\mathcal{J}\dot{z}-\lambda\nabla H(z)$
is compact. If $\Deg^\nabla_{S^1}(f_\lambda)\neq0$ 
then the equation \textup{($*$)}
has a~solution in $H_{2\pi\lambda}^1$.
\end{thm}

\begin{proof}
First note that if $f_\lambda^{-1}(0)$ is compact then
$f_\lambda$ is an element of $\mathcal{G}_{S^1}(E)$.
By the existence property, $\Deg^\nabla_{S^1}(f_\lambda)\neq0$
implies that $f_\lambda(z)=0$ for some $z\in E_1$. 
Hence a lift $\wt{z}\in H_{2\pi\lambda}^1$ 
of $z$ given by $\wt{z}(t)=z(\rho(t))$,
where $\rho\colon\R\to S^1$
is the standard covering projection,
is a solution of~($*$), which is our claim.
\end{proof}

\subsection{Applications to the Seiberg-Witten equations}
The description of the Seiberg-Witten equations
presented here is necessarily sketchy
(for more details we refer the reader to \cite{B,M,N,Sa}).
Let $M$ be a closed oriented Riemannian $3$-manifold.
A Spin$^{c}$-structure on $M$ consists of rank two Hermitian
vector bundle $S\to M$ called the \emph{spinor bundle}.
We write $\Omega^1(M,i\R)$ for the space of smooth
imaginary-valued $1$-forms on $M$ and $\Gamma(S)$
for the space of smooth cross-sections of the spinor bundle
$S\to M$. For each $a\in\Omega^1(M,i\R)$ there is an associated
\emph{Dirac operator} $D_a\colon\Gamma(S)\to\Gamma(S)$.

Recall that, in what follows, $d$ stands 
for the exterior derivative and $*$ denotes the Hodge star.
For a pair $(a,\varphi)\in\Omega^1(M,i\R)\oplus\Gamma(S)$
the \emph{Seiberg-Witten equations} are 
\[
\begin{cases}
D_a\varphi=0\\
*da=Q(\varphi),
\end{cases}
\]
where $Q(\varphi)\in\Omega^1(M,i\R)$ is a certain
quadratic form (nonlinear part of the equations).
The solutions of Seiberg-Witten equations 
are zeros of the \emph{Seiberg-Witten map}
\[\text{SW}\colon\Omega^1(M,i\R)\oplus\Gamma(S)
\to\Omega^1(M,i\R)\oplus\Gamma(S)\]
given by 
\[
\text{SW}(a,\varphi)=(*da-Q(\varphi),-D_a\varphi).
\]
After suitable Sobolev completion the Seiberg-Witten map
$\text{SW}$ can be written in the form $A-F$,
where $A=(*da,-D_a\varphi)$ is an unbounded self-adjoint
operator with a purely discrete spectrum 
and $F$ is a gradient map.
Moreover, the Seiberg-Witten map is equivariant 
for the action of the group $S^1$, 
which acts trivially on the component
arising from the differential forms 
and as complex multiplication
on the spinor component. 
It suggests that the $\text{SW}$ map
should fit to our abstract setting 
of the degree $\Deg^\nabla_{S^1}$.
Unfortunately, the set of zeros of the $\text{SW}$ 
map is not compact. However, we hope that
it is possible to reduce our problem to some
subspace of $\Omega^1(M,i\R)$ in such a way
that the reduced $\text{SW}$ map will have
a compact set of zeros,
which will be contained in the set 
of zeros of the original $\text{SW}$ map.
Verifying this claim is, however, still in progress.

\appendix
\section{} 
\label{sec:dodatekA}

Definition \ref{defn:deg} may be seen 
as a simple particular case
of a more general construction called 
the \emph{direct limit of a direct system of groups}.
Namely, for $i=0,1,\dotsc$ 
let $G_i$ denote an abelian group and 
$\alpha_i\colon G_i\to G_{i+1}$ a group homomorphism. 
With this notation we get the sequence
\[
G_0\stackrel{\alpha_0}{\longrightarrow}
G_1\stackrel{\alpha_1}{\longrightarrow}
G_2\stackrel{\alpha_2}{\longrightarrow}
G_3\to\dotsb
\]
Let $\wt{G}:=\coprod_{i=0}^{\infty}G_i$ denote a disjoint union, i.e.
\[
\wt{G}=\{(i,m)\mid i\in\N,\; m\in G_i\}.
\]
We introduce in $\wt{G}$ an equivalence relation. For $i>j$ we write
$(i,m)\sim (j,l)$ if
\[
\alpha_{i-1}\circ\dotsb\circ\alpha_{j+1}\circ \alpha_j(l)=m.
\] 
The \emph{direct limit of groups} is the set of equivalence classes
of the above relation, denoted by
\[
\varinjlim G_i=\wt{G}/\sim.
\]
Let $\varinjlim\fU(G)$ denote 
a direct limit of groups, where
\begin{itemize}
	\item $G_i=\fU(G)$ for all $i$,
	\item $\alpha_i$ is multiplication by an element
	 $a_i=\deg^\nabla_G(A_i,V^i)\in\fU(G)$.	
\end{itemize}
With this notation we can alternatively define our degree as a function
$\Deg^\nabla_{G}\colon\mathcal{G}_{G}(E)
\to\varinjlim\fU(G)\approx\fU(G)$ 
given by 
\[
\Deg^\nabla_{G}(f):=[(n,\deg^\nabla_G(f_n,U_n))]
\]
for $n$ large enough.

\end{document}